  \documentclass[12pt,reqno]{amsart}
  \usepackage{latexsym} 
  \usepackage[all]{xy}
  \usepackage{amsfonts} 
  \usepackage{amsthm} 
  \usepackage{amsmath} 
  \usepackage{amssymb}
  \usepackage{pifont}  
  \usepackage{enumerate}
  \xyoption{2cell}

 \usepackage{tikz}
\usetikzlibrary{arrows}

  \def\endproof{\hbox{$\sqcup$}\llap{\hbox{$\sqcap$}}\medskip} 
  \def\<{{\langle}} 
  \def\>{{\rangle}}

  \def\note#1{{}}

  \def\note#1{} 

  \def\cA{{\mathcal A}} 
  \def\cB{{\mathcal B}}

  \def\cO{{\mathcal O}}

   \def\ext#1#2#3#4{\mathrm{Ext}\sp{#1}\sb{#2}(#3,#4)}

  \def\tp{{\tilde{p}}}

  \def\beq{\begin{equation}} 
  \def\eeq{\end{equation}}

  \def\id{\mathrm{id}}

  \def\ot{{\otimes}}

 \def\htau{\hat{\tau}}

 \def\coker{\mathrm{coker}}

    \def \hp{\widehat{p}}

  \newcounter{zlist} 
  \newenvironment{zlist}{\begin{list}{(\arabic{zlist})}{ 
  \usecounter{zlist}\leftmargin2.5em\labelwidth2em\labelsep0.5em 
  \topsep0.6ex
  \parsep0.3ex plus0.2ex minus0.1ex}}{\end{list}}

  \newcounter{blist} 
  \newenvironment{blist}{\begin{list}{(\alph{blist})}{ 
  \usecounter{blist}\leftmargin2.5em\labelwidth2em\labelsep0.5em 
  \topsep0.6ex 
  \parsep0.3ex plus0.2ex minus0.1ex}}{\end{list}} 

  \newcounter{rlist}

\def\stac#1{\raise-.2cm\hbox{$\stackrel{\displaystyle\otimes}{\scriptscriptstyle{#1}}$}}

\def\cten#1{\raise-.2cm\hbox{$\stackrel{\displaystyle\widehat{\otimes}}
{\scriptscriptstyle{#1}}$}}

  \def\Label#1{\label{#1}\ifmmode\llap{[#1] }\else 
  \marginpar{\smash{\hbox{\tiny [#1]}}}\fi} 
  \def\Label{\label}

  \newtheorem{proposition}{Proposition}[section]
  \newtheorem{lemma}[proposition]{Lemma} 
  \newtheorem{corollary}[proposition]{Corollary} 
  \newtheorem{theorem}[proposition]{Theorem} 

\theoremstyle{definition} 
  \newtheorem{definition}[proposition]{Definition}
    
  \newtheorem{example}[proposition]{Example}

  \theoremstyle{remark} 
  \newtheorem{remark}[proposition]{Remark}

  \newcounter{c} 
   
  \newcommand{\etyk}[1]{\vspace{-7.4mm}$$\begin{equation}\Label{#1} 
  \addtocounter{c}{1}} 
  \renewcommand{\]}{\ifnum \value{c}=1 $$\else \end{equation}\fi} 
\def\ot{\otimes}

\def\CC{{\mathbb C}}

\def\KK{{\mathbb K}}

\def\NN{{\mathbb N}}

\def\PP{{\mathbb P}}

\def\WW{{\mathbb W}}

\def\ZZ{{\mathbb Z}}

\newcommand{\Cc}{\mathcal{C}}

\def\*C{{}^*\hspace*{-1pt}{\Cc}}

\def\text#1{{\rm {\rm #1}}}

 \def\1{\mathbf{1}}

      \def\tr{\mathrm{Tr}\ }

\def\bomega{\bar{\omega}}
 \def\proofof#1{{\sl Proof~of~#1.~~}} 
 \def\Weyl#1#2#3{\cB(#1;#2,#3)}
  \def\AWeyl#1#2{\cA(#1;#2)}

\begin{document}

\title{Circle and line bundles over generalized Weyl algebras}

\author{Tomasz Brzezi\'nski}
 \address{ Department of Mathematics, Swansea University, 
  Swansea SA2 8PP, U.K.} 
  \email{T.Brzezinski@swansea.ac.uk}   
 
\begin{abstract}
Strongly $\mathbb{Z}$-graded algebras or principal circle bundles and associated line bundles or invertible bimodules over a class of generalized Weyl algebras $\Weyl p q 0$ (over a ring of polynomials in one variable) are constructed. The Chern-Connes pairing between the cyclic cohomology of $\Weyl p q 0$ and the isomorphism classes of  sections of associated line bundles over $\Weyl p q 0$ is computed, thus demonstrating that these bundles, which are labelled by integers, are non-trivial and mutually non-isomorphic. The constructed strongly $\mathbb{Z}$-graded algebras are shown to have Hochschild cohomology reminiscent of  that of Calabi-Yau algebras. The paper is supplemented by an observation that a grading by an Abelian group in the middle of a short exact sequence is strong if and only if the induced gradings by the outer groups in the sequence are strong.
\end{abstract}
\maketitle

\section{Introduction}
This paper is devoted to the study of some aspects of the non-commutative geometry of degree-one {\em generalized Weyl algebras} \cite{Bav:gweyl} or {\em rank-one hyperbolic algebras} \cite{LunRos:Kas} over the polynomial ring in one variable (see Section~\ref{sec.pre.B}). These algebras can be interpreted as coordinate algebras of non-commutative surfaces, which are smooth  if the defining polynomial has no repeated roots. To be specific and fix the notation, we consider algebras $\Weyl p q r$ over a field $\KK$ of characteristic zero, labelled by $q,r\in \KK$, $q\neq 0$ and a non-zero polynomial $p$ in the variable $z$. The algebras $\Weyl p q r$ are  generated by $x,y,z$ subject to the relations
\begin{equation}\label{algebra}
xy = p(qz +r), \quad yx = p(z), \quad xz = (qz+r)x, \quad yz = q^{-1}(z -r)y.
\end{equation}
In the terminology of Bavula \cite{Bav:gweyl}, these are (all) degree-one generalized Weyl algebras over the ring $\KK[z]$. The polynomial $p$ is an element of $\KK[z]$ and $\sigma(z) = qz +r$ is an automorphism of $\KK[z]$, which form the defining data  a generalized Weyl algebra. The algebras $\Weyl p q r$ include examples of continuing and new interest in  non-commutative geometry such as quantum spheres \cite{Pod:sph}, non-commutative deformations of type-A Kleinian singularities \cite{Hod:def} and quantum weighted projective spaces \cite{BrzFai:tea}, \cite{Brz:Sei}.

The  main results of this article can be summarized as follows. When $r=0$ and 0 is a root of $p$  with multiplicity $k$, for each factorization $q = q_+q_-$ we construct a {\em strongly $\ZZ$-graded algebra}  \cite{NasVan:gra} (or a {\em principal $\KK\ZZ$-comodule algebra} \cite{BrzHaj:Che} or a  {\em quantum  principal circle bundle}) $\AWeyl p {q_\pm}^{(k)}$ with $\Weyl p q 0$ as its degree-zero part (see Theorem~\ref{thm.main}). The algebra $\AWeyl p {q_\pm}^{(k)}$ is the $k$-th Veronese subalgebra of a specific $\ZZ$-graded generalized Weyl algebra over the polynomial ring in two variables. As for every strongly $\ZZ$-graded algebra, the degree $n$ component ${\AWeyl p {q_\pm}^{(k)}}_{n}$ of $\AWeyl p {q_\pm}^{(k)}$ is a finitely generated projective  invertible bimodule over $\Weyl p q 0$, so it plays the role of sections of a line bundle over the non-commutative space represented by $\Weyl p q 0$. The idempotents for these modules can be explicitly described. We show that if 0 is a root of $p$ and $q$ is not a root of unity,  a  cyclic trace on $\Weyl p q r$ can be assigned to every non-zero root of $p$. Combining such a trace with the idempotents constructed earlier, the Chern-Connes pairing \cite{Con:non}  can be computed thus establishing that if $p$ has a root 0 (of multiplicity $k$) and another root, $r=0$ and $q$ is not a root of unity, then the Chern number of ${\AWeyl p {q_\pm}^{(k)}}_n$ is $-n$. Consequently, the $\Weyl p q 0$-modules ${\AWeyl p {q_\pm}^{(k)}}_n$ and ${\AWeyl p{ q_\pm,}^{(k)}}_m$ are not  isomorphic to each other if $m\neq n$ (see Theorem~\ref{thm.main}). This makes  ${\AWeyl p {q_\pm}^{(k)}}$ an example of what we term a {\em non-degenerate projectively graded algebra}. Finally, we look at homological properties of $\AWeyl p {q_\pm} ^{(k)}$. Exploring the results of \cite{Liu:hom} and employing methods of \cite{LiuWan:twi}, we calculate the Hochschild cohomology of $\AWeyl p {q_\pm}^{(1)}$ with values in its enveloping algebra, in the case in which $p$ has  no repeated roots. This turns out to be trivial except in degree 3, where it is isomorphic to $\AWeyl p {q_\pm}^{(1)}$ as a bimodule twisted by an algebra endomorphism of  $\AWeyl p {q_\pm}^{(1)}$ (see Proposition~\ref{prop.sus.a}). 

The paper is supplemented by an appendix in which it is shown that given an exact sequence of Abelian groups and an algebra graded by the middle group in the sequence, the grading is strong if and only if the induced gradings by the outer groups in the sequence are strong.

\section{Preliminaries}\label{sec.pre}\setcounter{equation}{0}
\subsection{Conventions and notation.} \label{sec.not}
All algebras considered in this paper are associative, unital algebras over a field $\KK$ of characteristic 0. Unadorned tensor products are over $\KK$.  The identity of a $\KK$-algebra $\cA$ is denoted by 1, while $\cA^{op}$ denotes the algebra with the multiplication  opposite to that of $\cA$, and $\cA^e = \cA\ot \cA^{op}$ is the enveloping algebra of $\cA$.

Given algebra endomorphisms $\mu,\nu$ of $\cA$ and an $\cA$-bimodule $M$, we write ${}^\mu\! M^\nu$ for the $\cA$-bimodule that is isomorphic to $M$ as a vector space, but has $\cA$-multiplications twisted by $\mu$ and $\nu$, i.e.\
\begin{equation}\label{twist}
a\cdot m \cdot b = \mu(a)m\nu(b), \qquad \mbox{for all} \; a,b\in \cA,\, m\in M.
\end{equation}
Note that if $\mu$ is an automorphism, then ${}^\mu\! A^\nu \cong A^{\nu\circ \mu^{-1}}$ as bimodules.

Given an automorphism $\sigma$ of an algebra $\cA$, the skew polynomial ring over $\cA$ in an indeterminate $z$ is denoted by $\cA[z;\sigma]$. This consists of all polynomials in $z$ with coefficients from $\cA$, subject to the relation $az = z\sigma(a)$, for all $a\in \cA$.

\subsection{$\ZZ$-graded algebras and non-commutative circle and line bundles.} \label{sec.pre.line}
A {\em $\ZZ$-graded algebra} is an algebra $\cA$ which decomposes into a direct sum of vector spaces $\cA = \oplus_{n\in \ZZ} \cA_n$ such that $\cA_m\cA_n \subseteq \cA_{m+n}$, for all $m,n\in \ZZ$. It follows immediately that $\cA_0$ is a subalgebra of $\cA$ and each of the $\cA_n$ is an $\cA_0$-bimodule by restriction of multiplication of $\cA$. A $\ZZ$-graded algebra $\cA$ is said to be {\em strongly graded}, if $\cA_m\cA_n = \cA_{m+n}$. As explained in  \cite[Section~AI.3.2]{NasVan:gra} a $\ZZ$-graded algebra $\cA$ is strongly graded if and only if there exist $\omega = \sum_i \omega_i'\ot \omega_i'' \in \cA_{-1}\ot \cA_1$ and $\bomega = \sum_i \bomega_i'\ot \bomega_i'' \in \cA_{1}\ot \cA_{-1}$ such that
\begin{equation}\label{conn.0}
\sum_i \omega_i'\omega_i'' = \sum_i \bomega_i'\bomega_i'' =1.
\end{equation}
Starting with such $\omega$, $\bomega$ one constructs inductively elements  
$\omega(n) \in \cA_{-n}\ot \cA_n$ as
\begin{equation}\label{conn.n}
\omega(0) = 1\ot 1, \qquad \omega(n) = \begin{cases}
\sum_i \omega_i'\omega(n-1) \omega_i''  & \quad \mbox{if $n>0$}, \\
\sum_i \bomega_i'\bomega(n+1) \bomega_i''  & \quad \mbox{if $n<0$} .
\end{cases}
\end{equation}
Due to \eqref{conn.0}, the evaluation of the multiplication on these elements yields the identity element of $\cA$. This implies that $\cA_0 = \cA_{n}\cA_{-n}$, for all $n$, and consequently $\cA_{m+n} = \cA_{m}\cA_{n}$ as required. 

The equality $\cA_0 = \cA_{n}\cA_{-n}$ translates into an $\cA_0$-bimodule isomorphism $\cA_0 \cong  \cA_{n}\ot_{\cA_0}\cA_{-n}$, therefore each of the $\cA_{n}$ is an invertible $\cA_{0}$-bimodule (with inverse $\cA_{-n}$), and hence in particular it is finitely generated and projective as a left and right $\cA_{0}$-module. More precisely, for the  $\omega(n)$ given by \eqref{conn.n}, let us write $\omega(n) =  \sum_{i=1}^N \omega'(n)_i \ot \omega''(n)_i$ and form an $N\times N$-matrix with entries 
\begin{equation} \label{idem}
E(n)_{ij} = \omega''(n)_i\omega'(n)_j.
\end{equation} 
 Since $\omega''(n)_i\in \cA_{n}$ and $\omega'(n)_j\in \cA_{-n}$, all  the  $E(n)_{ij}$ are degree-zero elements. Furthermore, \eqref{conn.0} guarantees that $E$ is an idempotent in the ring of $N\times N$-matrixes with entries from $\cA_0$. The left $\cA_0$-module isomorphism $\cA_n \to {\cA_0}^N E(n)$, where $N$ is the size of $E(n)$, is given by $a\mapsto (\sum_i a\omega'(n)_i E(n)_{ij})_j$. 
 
 Note in passing that due to the iterative definition of the $\omega(n)$, equations \eqref{conn.n} are equivalent to
 \begin{equation}\label{conn.n1}
\omega(0) = 1\ot 1, \qquad \omega(n) = \begin{cases}
\sum_i \omega(n-1)_i'\omega \omega_i(n-1)''  & \quad \mbox{if $n>0$}, \\
\sum_i \omega_i(n+1)'\bomega \omega_i(n+1)''  & \quad \mbox{if $n<0$} .
\end{cases}
\end{equation}
 
Strongly $\ZZ$-graded algebras are examples of {\em principal comodule algebras}, which in the context of non-commutative geometry play the role of coordinate algebras of principal fibre bundles \cite{BrzMaj:gau}, \cite{BrzHaj:Che}. In the $\ZZ$-graded case, the coacting Hopf algebra or the fibre is the group algebra of $\ZZ$ or, equivalently, coordinate algebra of the circle, hence strongly $\ZZ$-graded algebras represent circle principal bundles in non-commutative geometry. The elements $\omega(n)$ combine into a function from the group algebra of $\ZZ$ to $\cA\ot\cA$ which in the geometric context is interpreted as a {\em (strong) connection form} \cite{Haj:str}, \cite{DabGro:str}. In a general principal comodule algebra the existence of a (strong) connection form ensures that associated bundles are finitely generated projective modules with idempotents which in the circle bundle case take the form \eqref{idem}; see \cite{BrzHaj:Che}. Also in this case, bundles associated to one-dimensional representations of the circle group coincide with the $\cA_n$, and since the latter are invertible bimodules they can be given a genuine interpretation of sections of line bundles \cite{BegBrz:lin}.

 \subsection{The algebras $\Weyl p q r$ and other generalized Weyl algebras.}\label{sec.pre.B}
 Let $\cA$ be an algebra, $\sigma$ an automorphism of $\cA$ and $p$ an element of the centre of $\cA$. A {\em degree-one generalized Weyl algebra over $\cA$} is an algebraic extension $\cA(p,\sigma)$ of $\cA$ obtained by supplementing $\cA$ with additional generators $x,y$ subject to the following relations
 \begin{equation}\label{gW}
 xy = \sigma(p), \quad yx = p, \quad xa = \sigma(a)x, \quad ya = \sigma^{-1}(a)y;
 \end{equation}
see \cite{Bav:gweyl}. The algebras $\cA(p,\sigma)$ share many properties with $\cA$, in particular, if $\cA$ is a Noetherian algebra, so is $\cA(p,\sigma)$, and  if $\cA$ has no zero-divisors and $p\neq 0$, $\cA(p,\sigma)$ does not have zero-divisors either; see \cite{Bav:gweyl}. 

If $\cA = \KK[z]$, then every automorphism $\sigma$ of $\cA$ necessarily takes the form $\sigma(z) = qz +r$, for $q,r\in \KK$, $q\neq 0$. Hence any generalized Weyl algebra over  $\KK[z]$  coincides with an algebra $\Weyl p q r$ described in Introduction. A basis for $\Weyl p q r$  is given by monomials $x^kz^l$ and $y^kz^l$, $k,l\in \NN$. The Hochschild cohomology of the algebras $\Weyl p q r$ was computed in \cite{FarSol:Hoc}, \cite{SolSua:Hoc}, while  in  \cite{Liu:hom} it has been established that all the $\Weyl p q r$ are twisted Calabi-Yau algebras provided $p$ has no repeated roots. 

In the complex case $\KK=\CC$, $\Weyl p q r$ can be made into $*$-algebras by setting
$z^* = z$, $y^*=x$ as long as $q, r$ are real and $p$ has real coefficients. In the case $q \neq 1$ and $p(\frac{r}{1-q}) \geq 0$, $\Weyl p q r$ has one-dimensional $*$-representations $\pi_\lambda$, labelled by numbers $\lambda$ of modulus one, and given by 
$$
\pi_\lambda (z) = \frac{r}{1-q}, \qquad \pi_\lambda (x) = \lambda \sqrt{p\left(\frac{r}{1-q}\right)}.
$$
 Furthermore, if $q\in (0,1)$ and $r=0$, every real root  $\zeta$ of $p$ such that 
 $p(q^k\zeta)>0$, for all positive integers $k$,  defines an infinite-dimensional bounded $*$-representation $\pi_\zeta$ on the Hilbert space with orthonormal basis $e_k$, $k\in \NN$,
 $$
 \pi_\zeta(z) e_k = q^k\zeta e_k, \qquad \pi_\zeta(x) e_k = \sqrt{p\left(q^k\zeta\right)} e_{k-1}.
 $$
Quantum spheres \cite{Pod:sph} and quantum weighted projective spaces \cite{BrzFai:tea}, \cite{Brz:Sei} are examples of the $*$-algebras $\Weyl p q r$.

\section{Non-degenerate projectively graded algebras}\label{sec.proj.grad}\setcounter{equation}{0}
In this section we construct $\ZZ$-graded algebras, whose degree-zero part coincides with $\Weyl p q r$, and which 
fall within a specific category of graded algebras: 

\begin{definition}\label{def.proj.gra} 
~

\begin{zlist}
\item A $\ZZ$-graded algebra $\cA = \oplus_{n\in \ZZ} \cA_n$ is said to be {\em projectively graded} if, for all $n\in \ZZ$, $\cA_n$ is a projective left $\cA_0$-module.
\item A projectively graded algebra $\cA = \oplus_{n\in \ZZ} \cA_n$ is said to be {\em non-degenerate}, if 
\begin{blist}
\item for all $n\in \ZZ\setminus\{0\}$, $\cA_n$ is a finitely generated  left $\cA_0$-module;
\item for all $m,n\in \ZZ$, if $m\neq n$, then $\cA_m\not\cong \cA_n$ as left $\cA_0$-modules.
\end{blist}
\end{zlist}
\end{definition}

A strongly $\ZZ$-graded algebra $\cA$ is projectively graded but it is not necessarily non-degenerate (despite the fact that all the $\cA_n$ are finitely generated). There is an easy test which allows one to detect the failure of satisfying condition (2)(a) in Definition~\ref{def.proj.gra}.

\begin{lemma}\label{lem.trivial}
In a strongly $\ZZ$-graded algebra $\cA$, $\cA_n\cong \cA_0$ as left (resp.\ right) $\cA_0$-modules if and only if there exists a unit $u\in \cA_n$.
\end{lemma}
\begin{proof}
If $u$ is a unit in $\cA_n$, then $u^{-1}\in \cA_{-n}$, since $1\in \cA_0 = \cA_n \cA_{-n}$. This allows one to define mutually inverse isomorphisms of left $\cA_0$-modules by
$$
\phi: \cA_n \to \cA_0, \quad a\mapsto au^{-1}, \qquad \phi^{-1}: \cA_0 \to \cA_n, \quad b\mapsto bu.
$$

Conversely, given an isomorphism $\phi: \cA_n\to \cA_0$, let  $u = \phi^{-1}(1) \in \cA_n$. Since  $\cA$ is a strongly $\ZZ$-graded algebra, there exist elements $v_i\in \cA_{-n}$ and $u_i\in \cA_{n}$, such that $\sum_i v_iu_i =1$. Then, using the $\cA_0$-linearity of $\phi$ and the definition of $u$, one easily finds that 
$
u^{-1} := \sum_i v_i\phi(u_i)
$
is the inverse of $u$.
\end{proof}

\begin{corollary}\label{cor.trivial}
In a strongly $\ZZ$-graded algebra $\cA$ in which $\cA_0$ has the invariant basis number property, $\cA_n$ is a free $\cA_0$-module if and only if there is  a unit in $\cA_n$.
\end{corollary}
\begin{proof}
If $\cA_n \cong {\cA_0}^k$, then, since $\cA_{-n}$ is its inverse, also  $\cA_{-n} \cong  {\cA_0}^l$, for some $l$. Therefore,
$$
\cA_0 \cong {\cA_0}^k\ot_{\cA_0}{\cA_0}^l \cong {\cA_0}^{k +l-1},
$$
hence $k=l=1$ by the IBN property, and the assertion follows by Lemma~\ref{lem.trivial}.
\end{proof}

One way of proving that a strongly graded $\ZZ$-algebra $\cA$ is a non-degenerate projectively graded algebra is to use the Chern-Connes pairing between the cyclic cohomology and the $K$-theory of $\cA_0$. This method allows one not only to determine whether the components of $\cA$ are non-free but also to establish that they are not mutually isomorphic.  It was successfully applied recently in \cite{BrzFai:Hee} for algebras arising as total spaces of circle bundles over Heegaard quantum weighted projective lines, and earlier in \cite{Haj:bun} to prove that the coordinate algebra of the quantum group $SU_q(2)$ is a non-degenerate projectively graded algebra over the coordinate algebra of the quantum standard sphere $S^2_q$, and a similar statement for mirror quantum spheres \cite{HajMat:ind}. The method is based on evaluating a cyclic trace on $\cA_0$ at traces of idempotents for the $\cA_n$. As recalled in Section~\ref{sec.pre.line} there is an explicit formula for such idempotents. The construction of a suitable cyclic trace, i.e.\ one that detects differences between the  $\cA_n$, depends more heavily on the structure of $\cA_0$, and we will present it for the generalized Weyl algebras $\Weyl p q r$ described in the Introduction.

Consider the generalized Weyl algebra $\Weyl p q r$ given by generators and relations \eqref{algebra}, and assume that $q$ is not a root of unity. Let $\zeta$ be a root of $p$. Let us define  $\KK$-linear maps  $\htau_\zeta: \KK[z] \to \KK$ and $\tau_\zeta: \Weyl p q r \to \KK$, by
\begin{equation}\label{htau}
\htau_\zeta(z^n) = \frac{1}{1-q^n} \sum_{i=1}^n t^n_i\zeta^i, \quad \mbox{where} \quad t^n_n =1, \quad t^n_{n-k} = \sum_{i=1}^k\binom{n}{i} \frac{r^iq^{n-i}}{1-q^{n-i}}t^{n-i}_{n-k},
\end{equation}
$k=1,\ldots, n-1$, and 
\begin{equation}\label{tau0}
\tau_\zeta (x^mz^n) = \tau_\zeta (y^mz^n) = \begin{cases}
\htau_\zeta(z^n) &  \mbox{if $m=0$, $n\neq 0$}, \\
0  &  \mbox{otherwise} .
\end{cases} 
\end{equation}

\begin{lemma} \label{lemma.trace}
Let $p$ be a polynomial with coefficients in $\KK$ and with root $0$ and let $q$ be an element of $\KK$ that is not a root  of unity. Then, for any root $\zeta$ of $p$, the map $\tau_\zeta$ given by \eqref{tau} is a trace (cyclic cocycle) on $\Weyl p q r$. 
\end{lemma}

\begin{proof}
Since $\tau_0$ is the zero map, it is a trivial trace on $\Weyl p q r$. Let $\zeta$ be a non-zero root of $p$. By the definition of $\tau_\zeta$,
$$
\tau_\zeta (xz) = 0 = \tau_\zeta(zx), \qquad \tau_\zeta (yz) = 0 = \tau_\zeta(zy).
$$

Comparing respective powers of $\zeta$, one easily finds that
$$
\htau_\zeta(z^n) - \htau_\zeta((qz+r)^n) = \zeta^n, 
$$
therefore, for all polynomials $f\in \KK[z]$,
\begin{equation}\label{poly}
\htau_\zeta(f(z)) - \htau_\zeta(f(qz+ r)) = f(\zeta) -f(0).
\end{equation}
Next, let us consider the algebra automorphism $\sigma: \KK[z]\to \KK[z]$, $z\mapsto qz+r$, and define, for all $n$,
\begin{equation}\label{sn}
s_n(z) = \prod_{m=0}^{n-1} \sigma^{-m}(p(z)). 
\end{equation}
The relations \eqref{algebra} (cf.\ \eqref{gW}) imply that, for all $n\in \NN$,
$y^nx^n = s_n(z)$  and $x^ny^n = \sigma^n(s_n(z))$.
Since, for all $l=0,\ldots ,n-1$, the polynomial $\sigma^l\left({s}_n(z)\right)$ has a root $\zeta$,
the definition of $\tau_\zeta$ and equation \eqref{poly} yield
$$
\tau_\zeta (x^ny^n) = \htau_\zeta\left(\sigma^n\left({s}_n(z)\right) \right)
 = 
\htau_\zeta\left(\sigma^{n-1}\left({s}_n(z)\right) \right)
= \ldots = \htau_\zeta ({s}_n(z)) =  \tau_\zeta (y^nx^n).
$$
Hence $\tau_\zeta$ defined in \eqref{tau0} has the property  $\tau_\zeta(ab) = \tau_\zeta (ba)$, for all $a,b\in \Weyl p q r$, i.e.\ it is a cyclic cocycle (trace) on $\Weyl p q r$, as required.
\end{proof}

\begin{corollary} \label{cor.trace}
For $\Weyl p q 0$ in which  $q$ is not a root of unity, $p(0) = p(\zeta)=0$, $\zeta \neq 0$, the trace $\tau_\zeta$ \eqref{tau0} is explicitly given by
\begin{equation} \label{tau}
\tau_\zeta (x^mz^n) = \tau_\zeta (y^mz^n) =  \begin{cases}
\frac{\zeta^n}{1-q^n} &  \mbox{if $m=0$, $n\neq 0$}, \\
0  &  \mbox{otherwise} ,
\end{cases} 
\end{equation}
\end{corollary}
\begin{proof}
If $r=0$,  the inductive formula \eqref{htau} implies that for $t^n_{n-k}=0$ if $k\neq 0$, while $t^n_n=1$.
  \end{proof}
  
  \begin{remark}\label{rem.ref}
  The problem of finding all traces on $\Weyl p q r$ may be reduced to the problem of finding all $\KK$-linear maps $\KK[z]\to \KK$ that vanish on the commutator space of $\Weyl p q r$, i.e.\ on all elements
 $$
[x^nz^k,z^ly^n] =  \sigma^n\left(s_n(z)\right) \left(q^nz + \frac{q^n-1}{q-1}r\right)^{k+l} - s_n(z)z^{k+l},
$$
$k,l,n \in \NN$, where the $s_n(z)$ are defined in \eqref{sn}. The corresponding trace can then be defined as in equation \eqref{tau0}. I am grateful to the referee for pointing this out to me.
  \end{remark}
  
We are now in a position to construct quantum circle bundles over  $\Weyl p q 0$ and establish their non-degeneracy as projectively $\ZZ$-graded algebras.  Let us fix a non-zero polynomial  $p(z)$ with a root 0 of multiplicity $k$, and let $\tp(z)$ be obtained from $p(z)$ by factoring out $z^k$, i.e.\
 \begin{equation}\label{tp}
 p(z) = z^k\tp(z).
 \end{equation}
Next, we fix a non-zero $q\in \KK$ and  let $q_
\pm \in \KK$ be such that $q_+q_- = q$. With these data we associate an algebra $\AWeyl p {q_\pm}$ generated by $z_\pm$ and $x_\pm$ subject to the relations
 \begin{equation}\label{a.pm}
 \begin{gathered}
 z_+z_- =  z_-z_+, \qquad  x_+x_- = \tp( z_+z_-), \qquad  x_-x_+ = \tp(q z_+z_-),  \\ x_+z_\pm = q_{\pm}^{-1} z_\pm x_+, \qquad x_-z_\pm = q_{\pm} z_\pm x_-
 \end{gathered}
 \end{equation}
 The algebra $\AWeyl p {q_\pm}$ can be understood as a generalized Weyl algebra over the polynomial ring $\KK[z_+,z_-]$. We view it as a $\ZZ$-graded algebra with degrees of $z_\pm$ being equal to $\pm 1$, and that of $x_\pm$ being equal to $\pm k$. 
 Finally, we view the $k$-th Veronese subalgebra of  $\AWeyl p {q_\pm}$, 
$$
 \AWeyl p {q_\pm}^{(k)} := \bigoplus_{n\in \ZZ} \AWeyl p {q_\pm}_{nk},
$$
 as a $\ZZ$-graded algebra by considering $a\in \AWeyl p {q_\pm}^{(k)}$ to be of degree $n$ if it has a degree $kn$ in $\AWeyl p {q_\pm}$.
 
 \begin{theorem}\label{thm.main}
Let $p$ be a non-zero polynomial with a root 0 of multiplicity $k$.
\begin{zlist}
\item The $k$-th Veronese subalgebra ${\AWeyl p {q_\pm}^{(k)}}$  of  $\AWeyl p {q_\pm}$ is a strongly $\ZZ$-graded algebra with the degree-zero part equal to $\Weyl p q 0$, where $q=q_+q_-$. 
\item If $q$ is not a root of unity, then ${\AWeyl p {q_\pm}^{(k)}}$ is a non-degenerate projectively graded algebra if and only if  $p$ has at least one non-zero root.
\end{zlist}
\end{theorem}
\begin{proof}
(1) One easily verifies that the degree-zero part of $\AWeyl p {q_\pm}^{(k)}$ is generated by $x:= x_-z_+^k$, $y:=z_-^kx_+$ and $z:= z_+z_-$, and that these satisfy the relations \eqref{algebra}. We will continue to write $z$ for  $z_+z_-$.

Define the polynomial $\hp(z)$ by
\begin{equation}\label{hp}
\tp(z) = \tp(0) -z\hp(z).
\end{equation}
Note that $\tp(0) \neq 0$, so we can consider
\begin{equation}\label{omega.weyl}
\begin{aligned}
\omega &=  \frac{1}{\tp(0)^k} \left(q^k\hp(qz)^kz_-^k\ot z_+^k + x_- \ot \left(\sum_{i=0}^{k-1}z^i\hp(z)^i \tp(0)^{k-i-1}\right)x_+ \right)\\
&=\frac{1}{\tp(0)^k} \!\left(q^k\hp(qz)^kz_-^k\ot z_+^k + x_- \ot \frac{\tp(0)^k -z^k\hp(z)^k}{\tp(z)}x_+ \right) \in {\AWeyl p {q_\pm}^{(k)}}_{\!\!\!-1} \ot {\AWeyl p {q_\pm}^{(k)}}_{1}.
\end{aligned}
\end{equation}
A simple computation which uses the relations \eqref{a.pm},
\begin{eqnarray*}
\frac{1}{\tp(0)^k} \!\!\!\!\!&&\!\!\!\!\!\left(x_- \frac{\tp(0)^k -z^k\hp(z)^k}{\tp(z)}x_+ +q^k\hp(qz)^kz_-^kz_+^k\right)\\
&& = \frac{1}{\tp(0)^k} \left(\tp(qz)\frac{\tp(0)^k -q^kz^k\hp(qz)^k}{\tp(qz)} +q^k\hp(qz)^kz^k \right) =1,
\end{eqnarray*}
reveals that $\omega$ satisfies condition \eqref{conn.0}. In a similar way one checks that
$$
\bomega = \frac{1}{\tp(0)^k} \left(\hp(z)^kz_+^k\ot z_-^k + x_+ \ot \frac{\tp(0)^k -q^kz^k\hp(qz)^k}{\tp(qz)}x_- \right)\in {\AWeyl p {q_\pm}^{(k)}}_{1} \ot {\AWeyl p {q_\pm}^{(k)}}_{\!\!\!-1}
$$
satisfies the other condition in \eqref{conn.0}. Therefore, ${\AWeyl p {q_\pm}^{(k)}}$ is a strongly $\ZZ$-graded algebra as claimed.

(2)  If $p$ has no roots other than 0, then the second and third of relations \eqref{a.pm} imply that $x_\pm$ are units, hence, for all positive $n$,  $x_\pm^{n}$ are units in ${\AWeyl p {q_\pm}^{(k)}}_{\pm n}$, so ${\AWeyl p {q_\pm}^{(k)}}_{\pm n}\cong \Weyl p q 0$ by Lemma~\ref{lem.trivial}. Consequently ${\AWeyl p {q_\pm}^{(k)}}$ is a degenerate projectively graded algebra.

Assume that $p$ has non-zero roots. To prove that all the ${\AWeyl p {q_\pm}^{(k)}}_n$ are non-free (except for $n=0$) and mutually non-isomorphic we will pick a non-zero root $\zeta$ of $p$ and compute the value of the cyclic cocycle $\tau_\zeta$ \eqref{tau} on the trace of the idempotent $E(n)$ given by \eqref{idem}. We deal only with the positive $n$ case, the other case is similar. 

Let
$$
e_n  := \tr E(n) = \sum_i\omega(n)''_i\omega(n)'_i.
$$
Exploring \eqref{omega.weyl} and \eqref{conn.n1} as well as the defining relations \eqref{a.pm} of ${\AWeyl p {q_\pm}}_n$, we observe that $e_n$ is a polynomial in $z$ (independent of $x_\pm$), which, for non-negative $n$ is given by the following recursive formula:
\begin{equation}\label{e.1}
e_{n+1}(z)  = \frac{1}{\tp(0)^k} \left(\frac{\tp(0)^k -z^k\hp(z)^k}{\tp(z)}x_+e_n(z) x_- +q^kz_+^ke_n(z)\hp(qz)^kz_-^k\right).
\end{equation}
With the help of \eqref{a.pm} and \eqref{hp} this can be evaluated further to give
\begin{equation}\label{e.2}
e_{n+1}(z)  = \frac{1}{\tp(0)^k} \left(\left(\tp(0)^k - \hp(z)^kz^k\right)e_n(q^{-1}z) - \left(\tp(0)^k - q^k\hp(qz)^kz^k\right)e_n(z)\right)+e_n(z).
\end{equation}
In particular, 
\begin{equation}\label{e.3}
e_1(z) = \frac{1}{\tp(0)^k} \left(q^k\hp(qz)^k - \hp(z)^k\right)z^k+1,
\end{equation}
 hence $e_1(0)=1$. A simple inductive argument which uses \eqref{e.2} proves that in fact $e_n(0)=1$, for all positive $n$. We will show that $\tau_\zeta(e_n(z)) = -n$.
 
 Remembering the definitions of $\tau_\zeta$ and $\htau_\zeta$, specifically that $\htau_\zeta(1)=0$, we can use \eqref{e.3} to compute
 $$
 \tau_\zeta(e_1(z)) = \frac{1}{\tp(0)^k} \left(\htau_\zeta \left((qz)^k\hp(qz)^k - z^k\hp(z)^k\right)\right).
 $$
 Note that since $\zeta$ is a non-zero root of $p(z)$ it is also a root of $\tp(z)$, hence
 \begin{equation}\label{hpzeta}
 \zeta\hp(\zeta) = \tp(0).
 \end{equation}
Furthermore, $z\hp(z)$  is  a polynomial of degree at least one, so the formula \eqref{poly} can be applied thus yielding
 $$
  \tau_\zeta(e_1(z)) = -\frac{1}{\tp(0)^k} \left(\zeta\hp(\zeta)\right)^k = -1,.
  $$
 by \eqref{hpzeta}. Assume inductively that $\tau_\zeta(e_m(z)) = -m$. Then, using \eqref{e.2} and \eqref{poly} one finds
 \begin{eqnarray*}
  \tau_\zeta(e_{m+1}(z)) &=& \frac{1}{\tp(0)^k} \htau_\zeta\left(\left(\tp(0)^k - \hp(z)^kz^k\right)e_m(q^{-1}z) - \tp(0)^k\right)\\
  &&-  \frac{1}{\tp(0)^k} \htau_\zeta\left(\left(\tp(0)^k - q^k\hp(qz)^kz^k\right)e_m(z) - \tp(0)^k\right)+\htau_\zeta(e_m(z))\\
  &=& \frac{1}{\tp(0)^k}\left(\left(\tp(0)^k - \hp(\zeta)^k\zeta^k\right)e_m(q^{-1}\zeta) - \tp(0)^k\right) -m = -m-1.
  \end{eqnarray*}
Therefore, 
\begin{equation}\label{index}
\tau_\zeta(e_n(z)) = -n.
\end{equation}
 As explained in  \cite[Section~III.3]{Con:non} the formula \eqref{index} does not depend on the choice of the idempotent representing the isomorphism class of ${\AWeyl p {q_\pm}^{(k)}}_n$ in  the algebraic $K$-group $K_0({\AWeyl p {q_\pm}^{(k)}})$. The index formula \eqref{index}  determines the {\em Chern-Connes pairing} between even cyclic cohomology and the $K_0$-group. Since the numbers $\tau_\zeta(e_n(z))$ are not zero and distinguish between different $n$,  all the modules ${\AWeyl p {q_\pm}^{(k)}}_n$ are not free for all positive $n$ and they are mutually non-isomorphic.  Similar arguments confirm that the index formula \eqref{index} remains true also for negative values of $n$. Therefore, ${\AWeyl p {q_\pm}^{(k)}}$ is a non-degenerate projectively graded algebra.
\end{proof}

\begin{example}\label{ex.spindle}
The quantum weighted projective space or the quantum spindle algebra $\cO(\WW\PP_q(k,l))$ \cite{BrzFai:tea} is a generalized Weyl algebra $\Weyl p {q^{2l}} 0$, where
$$
p(z) = z^k \prod_{i=0}^{l-1}(1- q^{-2i}z).
$$
The  associated algebra $\AWeyl p {q_\pm}^{(k)}$, with $q_\pm =q^l$ coincides with the coordinate algebra of the quantum lens space $\cO(L_q(kl; k,l))$ \cite{HonSzy:len}. Hence the first part of Theorem~\ref{thm.main} recovers \cite[Theorem~3.3]{BrzFai:tea} in the case $k=1$ and \cite[Proposition~6.5]{AriKaa:Pim} for all other values of $k$. The index pairing calculations presented in the proof of the second part of Theorem~\ref{thm.main} confirm and extend those in \cite[Section~7.2]{AriKaa:Pim}.
\end{example}

\begin{remark}\label{rem.nonstrong}
Note that the algebra $\AWeyl p {q_\pm}$ is not strongly graded, unless, of course $k=1$. The degree-one submodule of $\AWeyl p {q_\pm}$ is generated by $x_\pm z_\mp^{k\mp 1}$ and $z_+$, while the degree minus one submodule is generated by $x_\pm z_\mp^{k\pm 1}$ and $z_-$. Every combination of the elements from the first group multiplied by the elements from the second will produce a term with at least one $z_\pm$. Such terms cannot combine to produce the identity element of $\AWeyl p {q_\pm}$.
\end{remark}

Let us recall that the Hochschild cohomology $H^l(\cA, M)$  of an algebra $\cA$ with values in an $\cA$-bimodule $M$ understood as a left $\cA^e$-module is defined as the right derived functor $\ext l {\cA^e} \cA M$ of the Hom-functor. Since $\cA^e$ is an $\cA^e$-bimodule, the Hochschild cohomology groups $\ext l {\cA^e} \cA {\cA^e}$ inherit an $\cA$-bimodule structure from the right $\cA^e$-module structure of $\cA^e$ (or the inner bimodule structure of $\cA^e$). 

\begin{proposition}\label{prop.sus.a}
Let $p(z)$ be a non-zero polynomial with a simple root 0, let $\tp(z)$ be given by \eqref{tp} with $k=1$, and let $\AWeyl p {q_\pm}$ be the algebra given by  generators and relations \eqref{a.pm}. If the polynomial $p$ has no repeated roots, then there exists  an algebra endomorphism $\kappa: \AWeyl p {q_\pm}\to \AWeyl p {q_\pm}$ such that
$$
H^l (\AWeyl p {q_\pm}, \AWeyl p {q_\pm}^e) \cong \begin{cases}
0, & \mbox{if $l\neq 3$}, \\
\AWeyl p {q_\pm}^{\kappa}, & \mbox{if $l=3$},
\end{cases}
$$ 
as $\AWeyl p {q_\pm}$-bimodules. The superscript $\kappa$ indicates twisting of the regular right $\AWeyl p {q_\pm}$-module structure as in \eqref{twist}.
\end{proposition}

\begin{lemma}\label{lemma.from}
Let $\cA$ be an algebra,
let $\omega$ be a normal regular element of $\cA$ (i.e.\ $\omega$ is not a zero-divisor and is such that $\cA \omega = \omega \cA$), and 
let $\cB:= \cA/\omega\cA$.
If there exists an algebra endomorphism $\nu: \cA\to \cA$ such that
\begin{equation}\label{hom.as}
\ext l  {\cA^e} {\cA} {\cA^e} = \begin{cases}
0, & \mbox{if $l\neq d+1$}, \\
\cA^{\nu}, & \mbox{if $l=d+1$},
\end{cases}
\end{equation}
as $\cA$-bimodules, then there exists an algebra endomorphism $\kappa :\cB \to \cB$, such that 
$$
\ext l {\cB^e} {\cB} {\cB^e} \cong \begin{cases}
0, & \mbox{if $l\neq d$}, \\
\cB^{\kappa}, & \mbox{if $l=d$},
\end{cases}
$$ 
as $\cB$-bimodules.
\end{lemma}

\begin{proof}
We will adapt the proof of \cite[Proposition~4.4]{Liu:hom}  (see also the proof of \cite[4.4.~Corollary]{Lev:som}) by incorporating a twist into it and extending it to this more general situation and calculate $\ext i {\cB^e} {\cB} {\cB^e}$. 

Let $\pi: \cA\to \cB$ be the canonical epimorphism.  Whenever we view a $\cB$-bimodule $M$ as an $\cA$-bimodule via $\pi$, we write $M_\pi$.

Since $\omega$ is not a zero divisor (a regular element of $\cA$), it defines an automorphism $\mu$ of $\cA$, by $a\omega = \omega \mu(a)$. The definition of $\cB$ is encoded in the following short exact sequence of $\cA$-bimodule maps
\begin{equation}\label{exact.seq}
\xymatrix{0  \ar[r] & \cA^{\mu^{-1}}\ar[r]^-{r_\omega} & \cA \ar[r] & \cB_\pi\ar[r] & 0,}
\end{equation}
where $r_\omega$ denotes the right multiplication by $\omega$. Since  $\ext i {\cA^e} {\cA} {\cA^e} =0$ unless $i=d+1$ in which case $\ext d {\cA^e} {\cA} {\cA^e} =\cA^{\nu}$,  the application of $\ext i  {\cA^e} {-} {\cA^e}$ to \eqref{exact.seq} yields a long exact sequence, whose nontrivial part is
$$
\xymatrix{0 \ar[r] & \ext {d+1}  {\cA^e} {\cB_\pi} {\cA^e} \ar[r] & \cA^{\nu}  \ar[r]^-{l_\omega} 
&  \cA^{\nu\circ \mu} \ar[r] & \ext {d+2}  {\cA^e} {\cB_\pi} {\cA^e} \ar[r] & 0,}
$$
where $l_\omega$ is the left multiplication by $\omega$. Since $l_\omega$ is a monomorphism and $\coker\ l_\omega = \cB$ we obtain
$$
\ext l  {\cA^e} {\cB_\pi} {\cA^e} = \begin{cases}
0, & \mbox{if $l\neq d+2$}, \\
\cB_\pi^{\nu\circ \mu}, & \mbox{if $l=d+2$}.
\end{cases}
$$
Note that $\cB^e = \cB\ot \cA^{op}/\cB\ot \cA^{op} (1\ot \omega)$ and $\cB\ot \cA^{op} = \cA^e/\cA^e(\omega \ot 1)$. Both $1\ot \omega$ and $\omega \ot 1$ are normal regular elements, inducing automorphisms $\id_\cB \ot \mu^{-1}$ and $\mu \ot \id_{\cA^{op}}$, respectively. This allows one to use the twisted version of the Rees lemma (see \cite[Lemma~1.2]{AjiSmi:inj} or \cite[Section~3.4]{Lev:som}) to compute the following chain of isomorphisms of $\cA$-bimodules
$$
\ext l {\cB^e} {\cB} {\cB^e}_\pi \cong {}^\mu \ext {l+1} {\cB\ot \cA^{op}} {\cB} {\cB\ot \cA^{op}} \cong {}^\mu \ext {l+2} {\cA^e} {\cB_\pi} {\cA^e}^{\mu^{-1}},
$$
where $\cB$ and its modules are always understood as  $\cA$-modules via the map $\pi$ and the indicated automorphisms of $\cA$. Therefore,
\begin{equation}\label{Gor}
\ext l {\cB^e} {\cB} {\cB^e}_\pi \cong  \begin{cases}
0, & \mbox{if $l\neq d$}, \\
{}^\mu \cB^{\nu}_\pi \cong \cB^{\nu\circ \mu^{-1}}_\pi, & \mbox{if $l=d$},
\end{cases}
\end{equation}
as $\cA$-bimodules. Let us write $\theta: \ext d {\cB^e} {\cB} {\cB^e}_\pi \to  \cB^{\nu\circ \mu^{-1}}_\pi$, for the isomorphism \eqref{Gor}. Since $\pi$ is an epimorphism, $\theta$ is an isomorphism of left $\cB$-modules $\ext d {\cB^e} {\cB} {\cB^e}\to  \cB$. The right $\cB$-module structures induce an endomorphism
$$
\kappa : \cB \to \cB, \qquad b\mapsto  \theta (\theta^{-1}(1) b),
$$
that makes $\theta$ an isomorphism of $\cB$-bimodules $\ext d {\cB^e} {\cB} {\cB^e}\to  \cB^\kappa$, as required.
\end{proof}

Note in passing that assumption \eqref{hom.as} is satisfied by {\em rigid Gorenstein} \cite{BroZha:dua}, hence in particular by {\em twisted Calabi-Yau} algebras \cite{Gin:Cal} of dimension $d+1$. 

\proofof{Proposition~\ref{prop.sus.a}}
Consider $\Weyl \tp q 0$, where $q=q_-q_+$,  and 
iterated skew polynomial  algebra over $\Weyl \tp q 0$, 
$$
\cA := \Weyl \tp q 0 [z_-;\sigma_-][z_+;\sigma_+],
$$
where $\sigma_-$ is an automorphism of $\Weyl \tp q 0$ and $\sigma_+$ of $\Weyl \tp q 0 [z_-;\sigma_-]$,
given by 
$$
\sigma_\pm(z)=z, \quad \sigma_\pm(x)=q_\pm x, \quad \sigma_\pm(y) = q_\pm^{-1}y, \quad \sigma_+(z_-)=z_-.
$$
Let $\omega = z_-z_+-z$. Note that $\omega$ is a normal regular element of $\cA$  and that  $\AWeyl p {q_\pm} \cong \cA/\cA\omega$ (the isomorphism sends $x_-$ to the class of $x$ and $x_+$ to the class of $y$). 

Since $\tp$ has no repeated roots, $\Weyl \tp q 0$ is a twisted Calabi-Yau algebra of dimension 2 by \cite[Theorem~4.5]{Liu:hom}. By \cite[Theorem~0.2]{LiuWan:twi} $\cA$ is a twisted Calabi-Yau algebra of dimension 4, which, in particular means that $\cA$ satisfies assumption \eqref{hom.as} with $d=3$, and the assertion follows by Lemma~\ref{lemma.from}.
\endproof

\appendix
\section{Strongly graded algebras and exact sequences of Abelian groups}
The notion of a  $\ZZ$-graded algebra is a special case of the notion of a group-graded algebra. Let $G$ be a group. A {\em $G$-graded algebra} $\cA$ decomposes into a direct sum of subspaces $\cA_g$  labelled by $g\in G$ such that $\cA_g\cA_h\subseteq A_{gh}$, for all $g,h\in G$. In case $\cA_g\cA_h= A_{gh}$, for all $g,h\in G$, $\cA$ is said to be {\em strongly $G$-graded}. We will write $|a |_G$ for the $G$-degree of $a\in \cA$.

Let $\cA = \oplus_{g \in G} \cA_g$ be a $G$-graded algebra. Any group epimorphism $\pi: G\to H$ induces an $H$-grading on $\cA$ by setting, for all $h\in H$,
$$
\cA_h = \bigoplus_{g\in \pi^{-1}(h)} \cA_g. 
$$
Given a group monomorphism $\varphi: K\to G$, one can define a $K$-graded algebra  
$$
\cA^{(K)} := \bigoplus_{k\in K} \cA_{\varphi(k)}.
$$

The aim of this appendix is to prove the following lemma, which supplements \cite[Section~A.I.3.1.b]{NasVan:gra}.
\begin{lemma}\label{lemma.tower}
Given a short exact sequence of Abelian groups,
$$
\xymatrix{ 0 \ar[r] & K \ar[r]^\varphi & G \ar[r]^\pi & H \ar[r] & 0,}
$$
a $G$-graded algebra $\cA$ is strongly graded if  and only if  the  induced $H$-grading on $\cA$  and $K$-grading on $\cA^{(K)}$ are strong. 
\end{lemma}
\begin{proof}
By  \cite[Section~A.I.3.1.b]{NasVan:gra}, if $\cA$ is strongly $G$-graded then also the induced graded algebras are strongly graded.

In the converse direction, let us take any $g\in G$. Since the $H$-grading on $\cA$ induced by $\pi$ is strong,
there exist 
$a_j, b_j\in \cA$ such that
$$
|a_j|_H = \pi(g) 
 \quad \mbox{and} \quad \sum_j a_jb_j =1.
$$
The exactness of the sequence implies that there exist $k_j\in K$, such that
$$
|a_j|_G = g + \varphi(k_j).
$$
Since  the $\varphi$ induced $K$-grading is strong, there exist $a_{ij}, b_{ij}\in \cA $ such that
$$
|a_{ij}|_K = -k_j
\quad \mbox{and} \quad \sum_i a_{ij}b_{ij} =1.
$$
Note that the first condition above means that $|a_{ij}|_G = -\varphi(k_j)$, hence 
$$
|a_ja_{ij}|_G =  |a_j|_G + |a_{ij}|_G = g \quad \mbox{and} \quad \sum_{i,j} a_ja_{ij}b_{ij}b_j = \sum_j a_jb_j =1,
$$
which implies that $\cA$ is a strongly $G$-graded algebra, as required.
\end{proof}

In the context of algebras discussed in the main body of this paper, one can take $\cA = \AWeyl p {q_\pm}$,  $G=K = \ZZ$, $K = \ZZ/k\ZZ$, the epimorphism $\pi: \ZZ\to \ZZ/k\ZZ$, $n\mapsto n \!\mod \! k$, and its kernel  monomorphism $\varphi: \ZZ \to \ZZ$, $n\mapsto kn$. Then the $\ZZ$-graded subalgebra induced  by $\varphi$ is $\cA_H = {\AWeyl p {q_\pm}^{(k)}}$. Since the $\ZZ$-grading of $\AWeyl p {q_\pm}$ is not strong (see Remark~\ref{rem.nonstrong}) and the $\ZZ$-grading of $\AWeyl p {q_\pm}^{(k)}$ is strong by Theorem~\ref{thm.main}, Lemma~\ref{lemma.tower} implies that the $\ZZ/k\ZZ$-grading of $\AWeyl p {q_\pm}$ is not strong.

\end{document}